\documentclass[11pt]{article}
\usepackage{amssymb, amsmath, amsthm, graphicx, hyperref}
\usepackage[left=1in,top=1in,right=1in]{geometry}
\date{}
\allowdisplaybreaks

\theoremstyle{definition}
\newtheorem{theorem}{Theorem}[section]
\newtheorem{lemma}[theorem]{Lemma}

\newtheorem{corollary}[theorem]{Corollary}

\counterwithin{equation}{section}

\title{Note on disjoint faces in simple topological graphs}
\author{Ji Zeng\thanks{Department of Mathematics, University of California San Diego, La Jolla, CA, 92093 USA. Supported by NSF grant DMS-1800746, ERC grant No.~882971 `GeoScape', and by the Erd\H{o}s Center. Email: {\tt jzeng@ucsd.edu}.}}

\begin{document}

\maketitle

\begin{abstract}
We prove that every $n$-vertex complete simple topological graph generates at least $\Omega(n)$ pairwise disjoint $4$-faces. This improves upon a recent result by Hubard and Suk. As an immediate corollary, every $n$-vertex complete simple topological graph drawn in the unit square generates a $4$-face with area at most $O(1/n)$. This can be seen as a topological variant of the Heilbronn problem for quadrilaterals. We construct examples showing that our result is asymptotically tight. We also discuss the similar problem for $k$-faces with arbitrary $k\geq 3$.
\end{abstract}

\section{Introduction}

A \emph{topological graph} is a graph drawn in the plane such that its vertices are represented by points, and its edges are represented by non-self-intersecting curves connecting the corresponding points. The curves are not allowed to pass through vertices different from their endpoints, and if two edges share an interior point, then they must properly cross at that point in common. A subgraph $H$ of a topological graph $G$ is said to be \emph{plane} if its edges do not cross. A \emph{$k$-face} generated by $G$ refers to the bounded open cell enclosed by a plane $k$-cycle in $G$. Here a $k$-cycle means a cycle of length $k$. We remark that a face can contain other vertices or parts of edges, and the boundary cycles of two disjoint faces can share vertices or edges. A topological graph is said to be \emph{simple} if each pair of its edges intersects at most once, either at a common endpoint or at a proper crossing point. It is easy to re-draw any topological graph with fewer crossings until it becomes simple (see e.g.~\cite{ringel1965sechsfarbenproblem}), hence the ``simple'' condition appears commonly, say, in the study of crossing numbers.

The main result of this paper is the following theorem. It improves the main result in a recent paper by Hubard and Suk \cite{HS} from a weaker bound of $\Omega(n^{1/3})$ in the same context.
\begin{theorem}\label{main}
    Every $n$-vertex complete simple topological graph generates at least $\Omega(n)$ pairwise disjoint $4$-faces.
\end{theorem}

This gives us the following immediate corollary.
\begin{corollary}\label{Heilbronn_upper}
    Every $n$-vertex complete simple topological graph drawn in the unit square generates a $4$-face with area at most $O(1/n)$.
\end{corollary}

We also show that Corollary~\ref{Heilbronn_upper}, and hence Theorem~\ref{main}, are asymptotically tight.
\begin{theorem}\label{Heilbronn_lower}
    For every integer $n \geq 1$, there is an $n$-vertex complete simple topological graph drawn in the unit square such that every face it generates has area at least $\Omega(1/n)$.
\end{theorem}

The famous Heilbronn problem \cite{roth1951problem} asks to determine the smallest $h(n)$ such that every set of $n$ points in the unit square forms a triangle with area at most $h(n)$. This problem is still open despite many efforts through more than 70 years. The current best known bounds for $h(n)$ are \begin{equation*}
    \Omega \left( \frac{\log n}{n^2} \right) < h(n)< O\left( \frac{1}{n^{7/6 - o(1)}} \right)
\end{equation*} due to Koml{\'o}s--Pintz--Szemer{\'e}di \cite{komlos1982lower} and Cohen--Pohoata--Zakharov \cite{cohen2024lower} respectively.

We call a topological graph \emph{geometric} if its edges are straight line segments. With this notion, the Heilbronn problem can be reformulated as follows: What is the smallest $h(n)$ such that every $n$-vertex complete geometric graph drawn in the unit square generates a $3$-face with area at most $h(n)$? Therefore, Corollary~\ref{Heilbronn_upper} and Theorem~\ref{Heilbronn_lower} can be seen as answers to a topological variant of the Heilbronn problem for quadrilaterals.

It is natural to consider the topological Heilbronn problem for triangles. However, Hubard and Suk \cite{HS} pointed out that we cannot expect a similar phenomenon for $k$-faces when $k$ is odd.
\begin{theorem}[Hubard--Suk]\label{Heilbronn_odd}
    For every integer $n \geq 1$ and every real number $\epsilon > 0$, there is an $n$-vertex complete simple topological graph drawn in the unit square such that every odd face it generates has area at least $1-\epsilon$.
\end{theorem}

In particular, Theorem~\ref{Heilbronn_odd} tells us that it is possible for all odd faces to have a nonempty intersection. On the other hand, we can still prove the existence of many pairwise disjoint $k$-faces when $k$ is even.
\begin{theorem}\label{main_general}
    For any fixed even number $k \geq 6$, every $n$-vertex complete simple topological graph generates at least $(\log n)^{1/4-o(1)}$ pairwise disjoint $k$-faces.
\end{theorem}

We conjecture that the bound of $(\log n)^{1/4-o(1)}$ in Theorem~\ref{main_general} can be improved to a bound of $\Omega(n)$. This will (if true) asymptotically match the counterpart given by Theorem~\ref{Heilbronn_lower}.

\section{Proof of Theorem~\ref{main}}

Our proof is based on a phenomenon of simple topological graphs discovered by Fulek and Ruiz-Vargas~\cite{fulek2013topological,ruiz2015empty}.

\begin{lemma}[Fulek--Ruiz-Vargas]\label{interior}
Let $G$ be a complete simple topological graph. If $H \subset G$ is a plane connected subgraph and $v\in G$ is a vertex not in $H$, then there exist at least two edges between $v$ and the vertices of $H$ that do not cross $H$.
\end{lemma}

The following lemma is the key to our proof. It is essentially proven by Hubard and Suk~\cite{HS}.

\begin{lemma}\label{key}
Let $G$ be a complete simple topological graph and $C \subset G$ be a plane $k$-cycle. If the face $F$ enclosed by $C$ contains at least $6k$ vertices of $G$ in its interior, then $G$ generates a $4$-face that lies inside (and different from) $F$.
\end{lemma}
\begin{proof}[Proof Sketch]
    When $k=3$, $F$ contains at least one vertex $v$ of $G$, and Lemma~\ref{interior} guarantees two edges from $v$ to $C$ that lie inside $F$, hence enclosing the desired $4$-face. This is already observed by Hubard and Suk, and combining this observation with a divide-and-conquer argument, they proved: if $k\geq 5$ and $F$ contains more than $6(k-4)$ vertices, then $G$ generates a $4$-face inside $F$ (Lemma~10 in \cite{HS}). When $k \geq 5$, their result implies what we want since $6k > 6(k-4)$.
    
    When $k=4$, we apply Lemma~\ref{interior} to any vertex $v$ of $G$ inside $F$, so we obtain two edges emanating from $v$ to $C$ that lie inside $F$. This divides $F$ into two faces $F_1$ and $F_2$. If $F_1$ and $F_2$ are $4$-faces, then we are done. Otherwise, we can assume without loss of generality that $F_1$ is a $3$-face and $F_2$ is a $5$-face. By our hypothesis and the pigeonhole principle, either $F_1$ contains $1$ vertex or $F_2$ contains $6k-1 = 23$ vertices, hence we finish the proof by reduction to previous cases.
\end{proof}

We need a fact by Garc{\'i}a, Pilz, and Tejel~\cite{GPT} whose proof also relies on Lemma~\ref{interior}.
\begin{lemma}[Garc{\'i}a--Pilz--Tejel]\label{maximalplane}
    In a complete simple topological graph, every plane subgraph is contained in another plane subgraph that is spanning and biconnected.
\end{lemma}

Our proof is inspired by a result of Aichholzer, Garc{\'i}a, Tejel, Vogtenhuber, and Weinberger~\cite{Graz}.

\begin{proof}[Proof of Theorem~\ref{main}]
Let $G$ be any $n$-vertex complete simple topological graph. A \textit{$4$-cell} refers to a (bounded) $4$-face generated by $G$ or an unbounded cell enclosed by a plane $4$-cycle of $G$. We consider a partial order on collections of pairwise disjoint $4$-cells of $G$. For two collections $\mathcal{C}_1$ and $\mathcal{C}_2$ of disjoint $4$-cells, we say $\mathcal{C}_1\preceq\mathcal{C}_2$ if $|\mathcal{C}_1| < |\mathcal{C}_2|$, or if $|\mathcal{C}_1| = |\mathcal{C}_2|$ and each $4$-cell in $\mathcal{C}_2$ is contained in some $4$-cell in $\mathcal{C}_1$. Let $\mathcal{C}$ be a maximal element with respect to this partial order.

Let $H$ be the subgraph of $G$ that consists of all the vertices and edges on the boundaries of the $4$-cells in $\mathcal{C}$. Notice that $H$ is plane. By Lemma~\ref{maximalplane}, we can augment $H$, by adding only edges of $G$, to another plane subgraph $H'\subset G$ such that $H'$ is biconnected and the vertex sets of $H$ and $H'$ coincide. Since $H'$ is plane and biconnected, it cuts the plane into cells $f_0,f_1,\dots,f_k$ each of whose boundaries is a plane cycle. Here, $f_0$ denotes the unbounded cell of this cutting.

Let $|f_i|$ be the number of vertices on the boundary of $f_i$, and $v(f_i)$ be the number of vertices in the interior of $f_i$. Suppose that $v(f_0) \geq 6|f_0|$, we consider $G$ as drawn on a sphere, choose a new pole outside $f_0$, and apply Lemma~\ref{key}, then we can conclude that $G$ generates a $4$-cell inside $f_0$. However, this contradicts $\mathcal{C}$ being maximal with respect to $\preceq$. Similarly, by the maximality of $\mathcal{C}$ and Lemma~\ref{key}, we must have\begin{equation*}
    v(f_i) < 6|f_i|.
\end{equation*} On the other hand, we have $\sum_i (|f_i|+v(f_i)) \geq n$, so $\sum_i |f_i| > n/7$.

Let $v(H')$ be the number of vertices of $H'$, and $e(H')$ be the number of edges of $H'$. We know that $2e(H') = \sum_i |f_i|$ by double-counting, so we have $e(H') > n/14$. Since $H'$ is plane, we have the classical bound $3v(H')-6 \geq e(H')$. Hence, we have $v(H)=v(H')> n/42$. Since each vertex of $H$ lies on the boundary of some cell in $\mathcal{C}$, and each cell in $\mathcal{C}$ has $4$ vertices on its boundary, \begin{equation*}
    |\mathcal{C}| \geq v(H)/4 > n/168.
\end{equation*} Notice that $\mathcal{C}$ contains at most one unbounded $4$-cell, and the others form a collection of pairwise disjoint $4$-faces. Therefore we conclude the proof.
\end{proof}

\section{Proof of Theorem~\ref{Heilbronn_lower}}

We describe a complete simple topological graph $D_n$ drawn in the unit square $[0,1] \times [0,1]$. Intuitively, $D_n$ is the outcome of this procedure: place $n$ vertices on a horizontal line, connect every pair by a lower semicircle, and ``inflate'' a small region to the lower left of each vertex. Our precise definition of $D_n$ is not optimized for the sake of simplicity, but one can easily make the area of every $k$-face arbitrarily close to $(k-2)/(n-2)$.

\begin{figure}[ht]
    \centering
    \includegraphics[scale=0.6]{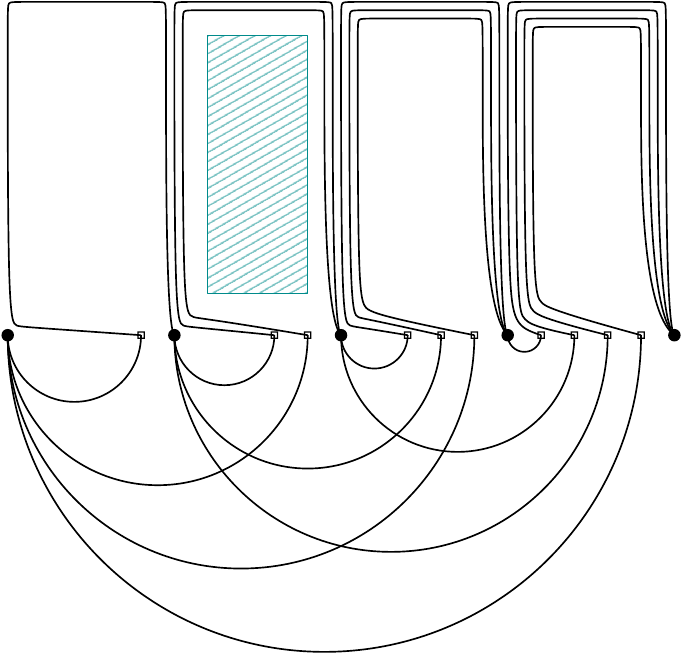}
    \caption{Illustration of $D_5$. The small square nodes represent the auxiliary points $v_{ij}$. The shaded rectangle represents one of the regions $R_i$.}
    \label{D5}
\end{figure}

We place $n$ vertices $v_i = (i/n , 1/2)$ for each $1\leq i\leq n$ on the equator of the unit square. We also pick auxiliary points $v_{ij} = (\frac{i}{n} - \frac{j}{n^2}, \frac{1}{2})$ for $1\leq j < i \leq n$. Inside the rectangular region $(\frac{i-1}{n},\frac{i}{n}) \times (\frac{1}{2},1)$, we draw a curve $c_{ij}$ from $v_i$ to $v_{ij}$ such that $c_{ij}$, together with the straight line segment between its endpoints, encloses a smaller region $R_i$ with area at least $\Omega(1/n)$. Moreover, these curves can be drawn without any crossings. See Figure~\ref{D5}. Now, for each $1\leq j < i \leq n$, we draw the edge $v_{i}v_{j}$ by concatenating the curve $c_{ij}$ (which ends at $v_{ij}$) and the lower semicircle from $v_{ij}$ to $v_{j}$. Since the crossings of edges only happen between their semicircle parts, we can check that this topological graph is simple hence finish the definition of $D_n$. It also follows from the crossing pattern of semicircles that two edges $v_iv_j$ ($i<j$) and $v_kv_{\ell}$ ($k<\ell$) cross in $D_n$ if and only if $i < k < j < \ell$ or $k < i< \ell < j$.

\begin{proof}[Proof of Theorem~\ref{Heilbronn_lower}]
    Let $C$ be a plane $k$-cycle in $D_n$. Using the crossing pattern of $D_n$, we can check that $C$ can be written as $v_{i_1} \to v_{i_2} \to \dots \to v_{i_k} \to v_{i_1}$ for some $i_1<i_2<\dots<i_k$. Therefore, the face enclosed by $C$ must contain the rectangles $R_{i_2}, R_{i_3},\dots, R_{i_{k-1}}$ each of whose areas is at least $\Omega(1/n)$, concluding the proof.
\end{proof}

\section{Proof of Theorem~\ref{main_general}}

In this proof, we shall consider two prototypes of simple topological graphs: convex geometric graphs and twisted graphs.

The complete \emph{convex geometric graph} $K_m$ is a complete simple topological graph whose vertices are $m$ points in convex position, and edges are straight line segments. We remark that any complete geometric graph must have its vertices in general position (i.e. no collinear triples). The vertices of $K_m$ can be labelled as $v_1,v_2,\dots, v_m$, such that two edges $v_iv_j$ ($i<j$) and $v_kv_{\ell}$ ($k<\ell$) cross if and only if $i < k < j < \ell$ or $k < i< \ell < j$. This means $K_m$ has the same crossing pattern as $D_m$ in the previous section.

The complete \emph{twisted graph} $T_m$ is a complete simple topological graph whose vertices can be labelled as $v_1,v_2,\dots, v_m$, such that two edges $v_iv_j$ ($i<j$) and $v_kv_{\ell}$ ($k<\ell$) cross if and only if $i < k < \ell < j$ or $k < i< j < \ell$. Specifically, for $j$ ranging from $1$ to $m$, we let $v_j$ be the point $(j,0)$ on the $x$-axis and draw the edges $v_iv_j$ ($i<j$) as follows: for $i$ ranging from $1$ to $j-1$, we draw a convex curve $c_{ij}$ from $v_j$ to a point $u_{ij}$ on the $y$-axis; we make sure the interior of the convex hull of $c_{ij}$ contains the plane origin, the points $v_1,v_2,\dots,v_{j-1}$, and all the previously drawn edges; we draw the edge $v_iv_j$ by concatenating the curve $c_{ij}$ and the line segment from $u_{ij}$ to $v_i$. See Figure~\ref{twisted}. We remark that Theorem~\ref{Heilbronn_odd} is proven by ``inflating'' the plane origin in the drawing of $T_m$.

    \begin{figure}[ht]
        \centering
        \includegraphics{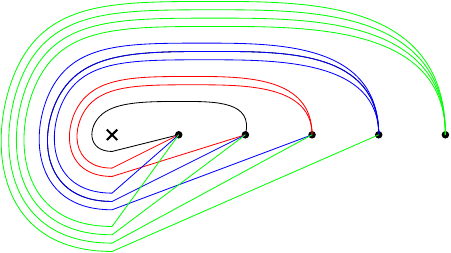}
        \caption{Illustration of $T_5$. The cross represents the origin of the plane.}
        \label{twisted}
    \end{figure}

Our proof is based on a Ramsey-type result on the crossing patterns of simple topological graphs by Suk and Zeng \cite{suk2022unavoidable} (see also \cite{pach2003unavoidable}). Two simple topological graphs $G$ and $H$ are said to be \emph{weakly isomorphic} if there is an incidence preserving bijection between $G$ and $H$ such that two edges of $G$ cross if and only if the corresponding edges in $H$ cross as well.

\begin{theorem}[Suk--Zeng]\label{unavoidable}
Every $n$-vertex complete simple topological graph has an induced subgraph on $m\geq (\log n)^{1/4 - o(1)}$ vertices that is weakly isomorphic to either $K_m$ or $T_m$.
\end{theorem}

\begin{proof}[Proof of Theorem~\ref{main_general}]
    Let $G$ be any $n$-vertex complete simple topological graph. By Theorem~\ref{unavoidable}, there are $m \geq (\log n)^{1/4 - o(1)}$ vertices of $G$ inducing a subgraph $H$ that is weakly isomorphic to either the convex geometric graph $K_m$ or the twisted graph $T_m$.

    Suppose a simple topological graph $H'$, weakly-isomorphic to $H$, generates a collection $\mathcal{C}$ of pairwise disjoint $k$-faces, we claim that $H$ also generates at least $|\mathcal{C}| - 1$ pairwise disjoint $k$-faces. Indeed, the boundary cycles appearing in $\mathcal{C}$ must not cross each other in the drawing of $H'$, hence the same holds in $H$ by weak isomorphism. Moreover, for the boundary cycle $C$ of an arbitrary element in $\mathcal{C}$, all vertices appeared in $\mathcal{C}$, other than those of $C$, are in the same side of $C$ in the drawing of $H'$. The same property also holds in $H$, because whether $u$ and $v$ are on the same side depends on the parity of the number of crossings between the curve $uv$ and the cycle $C$, and this quantity is preserved under weak isomorphism. Therefore, every such cycle $C$ encloses a cell in $H$ not containing any vertices appearing in $\mathcal{C}$. At most one such cell can be unbounded, and the remaining form pairwise disjoint faces.

    It suffices for us to show that both $K_m$ and $T_m$ generate $\Omega(m)$ many pairwise disjoint $k$-faces. For the convex geometric graph $K_m$, we can label its vertices as $v_1,v_2,\dots,v_m$ in a non-decreasing order of
    their $x$-coordinates. Then we divide the vertex set into $\Omega(m)$ many $k$-tuples \begin{equation*}
        \tau_i=\{ v_{(i-1)k+1},\ v_{(i-1)k+2},\ \dots,\ v_{ik} \} \quad \text{for} \quad 1\leq i\leq \lfloor m/k \rfloor.
    \end{equation*} Because $K_m$ is a convex geometric graph, the interior of the convex hull of each $\tau_i$ is a $k$-face. Furthermore, the projections of these $k$-faces on the $x$-axis are pairwise disjoint open intervals, so $K_m$ generates $\Omega(m)$ many pairwise disjoint $k$-faces.

    For the twisted graph $T_m$, we recall, from the beginning of this section, that the vertices of $T_m$ are labelled as $v_1,v_2,\dots,v_m$. We divide the vertex set into $\Omega(m)$ many $k$-tuples $\tau_i$ as we did for $K_m$. For a particular $k$-tuple (for example) $\tau_1$, we define its corresponding cycle\begin{equation*}
        C_1 = v_{1}\to v_{\frac{k}{2}+1}\to v_{k}\to v_{\frac{k}{2}}\to v_{k-1}\to v_{\frac{k}{2}-1}\to v_{k-2}\to v_{\frac{k}{2}-2}\to \dots\to v_{\frac{k}{2}+2}\to v_{2}\to v_{1}.
    \end{equation*} Using the crossing pattern of $T_m$, we can check that $C_1$ is plane and hence it encloses a $k$-face $F_1$. Similarly we can define the cycle $C_i$ and the face $F_i$ according to $\tau_i$ for all $1\leq i\leq \lfloor m/k \rfloor$. We claim that any two faces $F_i$ and $F_j$ ($i<j$) are disjoint. Indeed, we use $R$ to denote the convex hull of the edge $v_{ik}v_{(j-1)k+1}$. By the process of drawing $T_m$, the induced subgraph on $\tau_i$ is inside $R$, and the induced subgraph on $\tau_j$ is outside $R$. As a consequence, $F_i$ is contained in $R$, and $F_j$ either contains $R$ or does not intersect $R$. Notice that the negative portion of the $x$-axis is a ray emanating from the plane origin that crosses $C_j$ an even number of times. This means the plane origin is outside the face $F_j$ enclosed by $C_j$. However, the plane origin is contained inside $R$. So $F_j$ does not intersect $R$ and hence is disjoint from $F_i$. Therefore, we conclude that $T_m$ generates $\Omega(m)$ many pairwise disjoint $k$-faces.
\end{proof}

\medskip

\noindent {\bf Acknowledgement.} I wish to thank Andrew Suk for reading an early draft, Alexandra Weinberger for discussions at SoCG 2023, and anonymous referees for several suggestions. This paper is dedicated to the memory of Emily Zhu.

\bibliographystyle{abbrv}
{\footnotesize\bibliography{main}}

\end{document}